\documentclass[12pt]{article}
\usepackage[utf8]{inputenc}
\usepackage[english]{babel}
\usepackage[dvips]{graphicx}
\usepackage{indentfirst}
\usepackage{latexsym}
\usepackage{amsthm}
\usepackage{amsfonts}
\usepackage{amssymb}
\usepackage{amsmath}
\usepackage{hyperref}

\theoremstyle{plain}
\newtheorem{theorem}{Theorem}

\theoremstyle{definition}

\newtheorem{remark}{Remark}

\newtheorem{property}{Property}

\begin{document}

\noindent UDC 519.17

\title{Algorithm for Constructing Related Spanning Directed Forests of Minimum Weight} 
\author{V.\,A. Buslov}
\begin{center}
{\bf Algorithm for Constructing Related Spanning Directed Forests of Minimum Weight} 
\end{center}
\begin{center}
{\large V.\,A. Buslov}
\end{center}

\begin{abstract} 
An algorithm is proposed for constructing directed spanning forests of the minimum weight, in which the maximum possible degree of affinity between the minimum forests is preserved when the number of trees changes. The correctness of the algorithm is checked and its complexity is determined, which does not exceed $ O (N ^ 3) $ for dense graphs. The result of the algorithm is a set of related spanning minimal forests consisting of $ k $ trees for all admissible $ k $. 
\end{abstract}

This article is based on the results of \cite{V6}, and the proposed algorithm is a significant modification of the \cite{V8} algorithm. The notations and definitions correspond to those adopted in these two works, up to some specially specified simplifications.

\section{Basic notations and definitions}

The term "graph"\ is used for both undirected and directed graphs, unless this leads to confusion.

For a digraph $G$, we denote the set of its vertices by ${\cal V}G$, and the set of arcs by ${\cal A}G$.
 
The original object is a weighted digraph $V$, with a set of vertices ${\cal V}V={\cal N}$, $|{\cal N}|=N$; the arcs $(q,p)\in{\cal A}V$ are assigned real weights $v_{qp}$.  We consider spanning subgraphs (with a set of vertices ${\cal N}$) that are entering forests. An entering (incoming) forest is a digraph in which at most one arc comes from each vertex and there are no cycles. A tree is a connected component of the forest. The root of a tree (forest) is the vertex from which arc does not come. Let $T^F_q$ denote the tree of the forest $F$ rooted at vertex $q$. The set of roots of $F$ is denoted by ${\cal K}_F$. The outgoing forest is obtained by replacing all the arcs of the entering forest with inverse ones (then the root is the vertex into which arcs do not go). Only entering forests are used; in the future, simply forests where this is not required to be specifically specified.

The subgraph $H$ of graph $G$ induced by the set ${\cal S}\subset{\cal V}G$ (or the restriction of the graph $G$ to the set ${\cal S}$) is a subgraph whose ${ \cal V}H={\cal S}$, and the set of its arcs are all arcs of the graph $G$, both ends of which belong to the set ${\cal S}$. For it we use the notation $H=G|_{\cal S}$.

If there is an arc whose outcome belongs to the set ${\cal S}$, but whose entry does not belong, then we say that the arc comes from the set ${\cal S}$. Similarly, if there is an arc whose entry belongs to ${\cal S}$ and whose outcome does not belong, then we say that the arc enters ${\cal S}$. The outgoing neighborhood ${\cal N}^{out}_{\cal S}(G)$ of the set ${\cal S}$ is the set of entries of arcs outgoing in graph $G$ from the set ${\cal S }$; The incoming neighborhood ${\cal N}^{in}_{\cal S}(G)$ is defined similarly.
 
For a subgraph $G$ of  graph $V$ and a subset of the vertex set ${\cal S}\subseteq {\cal N}$, we introduce weights 
\begin{equation}
\Upsilon^G_{\cal S}=\sum_{\begin{smallmatrix}q\in{\cal S} \\ (q,p)\in {\cal A}G\end{smallmatrix}} v_{qp} \ , \ \ \Upsilon^G=\Upsilon^G_{\cal N}=\sum_{(q,p)\in {\cal A}G} v_{qp} \ . \label{ves}
\end{equation}
The value $\Upsilon^G_{\cal S}$ is also formed from arcs whose entries do not belong to the set ${\cal S}$.  It is with this definition that the weight additivity property is satisfied: $\Upsilon^G_{{\cal S}\cup{\cal S'}}=\Upsilon^G_{\cal S}+\Upsilon^G_{\cal S'}$ if  ${\cal S}\cap{\cal S'}=\emptyset$. If in $G$ arcs do not come from the set ${\cal S}$ itself, then $\Upsilon^G_{\cal S}=\Upsilon^{G|_{\cal S}} $.

${\cal F}^k$ is a set of spanning forests consisting of $k=1,2,\ldots ,N$ trees.  We denote the minimum weight $\Upsilon^F$ among all forests $F\in{\cal F}^k$ by $\phi^k$:

$$ \phi^k=\min_{F\in{\cal
F}^k}\Upsilon^F \ . $$

If ${\cal F}^k=\emptyset$, then we set $\phi^k=\infty$, in particular, $\phi^0=\infty$. The set ${\cal F}^N$ consists of only the empty forest and $\phi^N=0$. Any forest from ${\cal F}^{N-1}$ contains exactly one arc, so $\phi^{N-1}=\min\limits_{(q,p)\in{\cal A}V }v_{qp}$.

$\tilde{\cal F}^k$ is a subset of the set of forests ${\cal F}^k$ on which the minimum $\phi^k$ is achieved: $
F\in\tilde{\cal F}^k\Leftrightarrow F\in{\cal F}^k$ and $\Upsilon^F=\phi^k$. Forests from $\tilde{\cal F}^k$ are called minimal.

$F^G_{\uparrow{\cal S}}$ --- a graph obtained from $F$ by replacing arcs coming from the vertices of the set ${\cal S}$ with arcs coming from the same vertices in the graph $G $. 

For any subset ${\cal S}\subset {\cal N}$, its complement $\overline{\cal S}={\cal N}\setminus {\cal S}$.

We introduce the remaining necessary definitions and notations as needed.

\section{Objects and properties used}

\subsection{Arc replacement operation}

We will need some properties of the arc replacement operation.
The following property \cite[Corollary 6 from Lemma 1]{V6} of the arc replacement operation is valid if $F$ and $G$ are forests. 

\begin{property}{\it Let $F$ and $G$ be forests with the same vertex set, and let $T^F$ and $T^G$ be trees of $F$ and $G$ respectively. Assume that ${\cal D}\subset{\cal V}T^F\cap{\cal V}T^G$, ${\cal N}^{in}_{\cal D}(G)=\emptyset$ and ${\cal N}^{out}_{\cal D}(G)\subset{\cal V}T^G\setminus{\cal V}T^F $. Then both $F_{\uparrow\cal D}^G$ and $G_{\uparrow\cal D}^F$ are forests.}
\end{property}

If forests $F$ and $G$ are minimal, then \cite[Lemma 3]{V6}
is valid 
\begin{property}
	{\it Let $F\in\tilde{\cal F}^{k+1}$ and $G\in\tilde{\cal F}^k$, and let ${\cal D}$ --- a subset of the vertex set such that $F^G_{\uparrow\cal D} \in {\cal F}^k$ and $G^F_{\uparrow\cal D}\in {\cal F}^{k+ 1}$. Then $F^G_{\uparrow\cal D} \in \tilde{\cal F}^k$ and $G^F_{\uparrow\cal D}\in \tilde{\cal F}^{k+1 }$.}
\end{property}
\subsection{Different minima on subsets of the set of vertices}

In \cite{V8}, for any ${\cal S}\subset{\cal N}$, a set of special tree-like minima is defined:
\begin{equation}
 \lambda_{\cal S}^{\bullet q}=\min_{T\in {\cal T}^{\bullet q}_{\cal S}}\Upsilon^T, \  \lambda_{\cal S}^\bullet= \min_{q\in{\cal S}} \lambda_{\cal S}^{\bullet q},  
\label{bt}
\end{equation} 
where ${\cal T}^{\bullet q}_{\cal S}$ is the set of trees with vertex set ${\cal S}$ rooted at $q\in {\cal S}$. 

Forest-like minima are also defined in \cite{V8}
   
\begin{equation}
 \mu_{\cal S}^{\bullet q}=\min_{F\in {\cal F}^{\bullet q}_{\cal S}}\Upsilon^F_{\cal S}\ , \ \ \mu_{\cal S}^\bullet= \min_{q\in{\cal S}} \mu_{\cal S}^{\bullet q}, 
\label{bf}
\end{equation} 
where ${\cal F}^{\bullet q}_{\cal S}$ is the set of spanning forests in which the set ${\cal S}$ contains exactly one root --- vertex $q\in{ \cal S}$. 
Arcs coming from the vertices of $\overline{\cal S}$ in any forest $F\in {\cal F}^{\bullet q}_{\cal S}$ are arbitrary and do not affect the value of weight $\mu^{\bullet q}_{\cal S}$.

For these quantities the following is true

\begin{property}\cite[Proposition 2]{V8}
{\it Let $F\in\tilde{\cal F}^k$ and $s\in{\cal K}_F$, 
then }
\begin{equation}
\mu_{\cal S}^\bullet = \mu_{\cal S}^{\bullet s}= 
\lambda_{\cal S}^\bullet = \lambda_{\cal S}^{\bullet s} =\Upsilon^F_{\cal S} \ ,  
\label{mlu}
\end{equation}
{\it where}  ${\cal S}={\cal V}T^F_s$.   
\end{property}

Another one forest-like minimum option is also defined \cite{V8}:
\begin{equation}
\mu^\circ_{{\cal S}} = \min_{F\in{\cal F}^\circ_{{\cal S}}} \Upsilon^F_{\cal S} \ , 
\label{cf} 
 \end{equation} 
where ${\cal F}^{\circ}_{\cal S}$ is the set of spanning forests in which arcs come from all vertices of the set ${\cal S}$. And also arcs coming from the vertices of $\overline{\cal S}$ in any forest $F\in {\cal F}^\circ_{\cal S}$  are arbitrary and do not affect the value of weight $\mu^\circ_{\cal S}$. 

\section{Constructing related forests from pseudo-related ones}

\subsection{Pseudo-related and related forests}

Spanning forests with varying numbers of trees may share some similar features. Pseudo-related forests \cite{V8} demonstrate sufficient similarity. 

We call the forest $F\in{\cal F}^{k+1}$ {\it pseudo-ancestor} of the forest $G\in{\cal F}^{k}$, and $G$ {\it pseudo-descendant} of $F$ if there is a root $y$ of $F$ such that: ${\cal K}_G={\cal K}_F\setminus \{ y \}$ and $G|_{{\cal V}T^F_q}=T^F_q$ for $q\in{\cal K}_G$. We call such forests pseudo-related, and the set of such pseudo-descendants of $F$ is denoted by ${\cal P}^F_y$.

Since $F$ and its pseudo-descendant $G\in{\cal P}^F_y$ differ only in the arcs coming from vertices of the only tree of $F$ (the tree rooted at vertex $y$), then
\begin{equation}
\Upsilon^G-\Upsilon^F=\Upsilon^G_{\cal Y}-\Upsilon^F_{\cal Y} \ , \ \ \ {\cal Y}={\cal V}T^F_y \ . 
\label{ps}
\end{equation}

The greatest similarity is demonstrated by related forests \cite{V6}, which, taking into account the definition of pseudo-related, can be defined as follows.

We call the forest $F\in{\cal F}^{k+1}$ the {\it ancestor} of the forest $R\in{\cal F}^{k}$, and $R$ {\it the descendant} of the forest $ F$, if for some $y\in{\cal K}_F$ the following holds: $R\in{\cal P}^F_y$ and $R|_{{\cal V}T^F_y}$ is a tree. The forests $F$ and $R$ are called related.
In fact, this means that in $R$ there is only one arc coming from the set ${\cal V}T^F_y$. The entry of this arc belongs to some tree $T^F_x$. So: ${\cal K}_R={\cal K}_F\setminus \{ y \}$; $T^R_q=T^F_q$ for $q\in{\cal K}_R\setminus \{ x\}$; $R|_{{\cal V}T^F_x}=T^F_x$.  This corresponds to the definition of \cite{V6}, which does not use the concept of pseudo-relatedness. 

As shown in Fig. \ref{poda} for pseudo-descendant $G\in {\cal P}^F_y$ , generally speaking, several arcs come from the set ${\cal V}T^F_y$ to different trees of $F$, including, possibly, several arcs come into one tree, and even several arcs come into one vertex. The descendant has a unique arc coming from the set ${\cal V}T^F_y$. Its entry belongs to the set of vertices of a certain tree ${\cal V}T^F_x$ (see Fig. \ref{pods}).  

If $R\in{\cal P}^F_y$ is also a descendant of the forest $F$, then (\ref{ps}) still holds for it, with $G$ replaced by $R$. If $x$ is the root of the forest tree $R$ that “digested” the tree $T^F_y$ (${\cal V}T^R_x =  {\cal V}T_x^F\cup{\cal V}T_y^F$), then completed
\begin{equation}
\Upsilon^R_{\cal Z}=\Upsilon^F_{\cal X}+\Upsilon^R_{\cal Y} \ , \ \ {\cal Y}={\cal V}T^F_y , \   {\cal X}={\cal V}T^F_x  , \ \  {\cal Z}={\cal X}\cup{\cal Y} \ . 
\label{p}
\end{equation}

Let us formulate the main theorem \cite[Theorem 2 (on  related forests)]{V6} as the property

\begin{property} {\it Let ${\cal F}^{k}\neq \emptyset$, then any forest from $\tilde{\cal F}^{k+1}$ has a descendant in $\tilde{\cal F}^{k }$ and vice versa --- any forest from $\tilde{\cal F}^{k}$ has an ancestor in $\tilde{\cal F}^{k+1}$.} 
\end{property}

This property indicates the minimum possible changes that must be made to a minimal spanning forest with some number of trees in order to obtain a minimal spanning forest with the number of trees differing by one.

Creating a direct algorithm for constructing  $R\in\tilde{\cal F}^{k}$, which is a minimal descendant of the forest $F\in\tilde{\cal F}^{k+1}$, encounters an additional difficulty. It is necessary to determine which vertex of the "modified" \ tree $T^F_y$ should be the root of the tree $R|_{{\cal V}T^F_y}$.  
If you simply iterate through the vertices of the set ${\cal V}T^F_y$, assigned by the roots of the tree 
$R|_{{\cal V}T^F_y}$, this will inevitably increase the complexity of the algorithm.  
If the task is simply to build a minimal forest from the set $\tilde{\cal F}^{k}$ with an existing forest from $\tilde{\cal F}^{k+1}$, there is no need to track such a refined relationship.
However, for practical purposes it is just the opposite: the more interesting and important situation  under  constructing minimal forests is when the forest structure is preserved to the maximum possible extent. The above problem of determining this virtual root (the root of tree $R|_{{\cal V}T^F_y}$) can be circumvented.

\begin{figure}[h]
\unitlength=1mm
\begin{center}
\begin{picture}(100,41)

\put (10,35){${\cal V}T^F_y$}
\put (80,35){${\cal V}T^F_x$}

\put (39,12){$y$}
\put (60,7){$w$}
\put (70,7){$u$}
\put (90,17){$x$}

\put(10,10){$\centerdot$}
\put(10,30){$\centerdot$}
\put(20,10){$\centerdot$}
\put(20,30){$\centerdot$}
\put(50,10){$\centerdot$}
\put(50,30){$\centerdot$}
\put(60,10){$\centerdot$}
\put(60,30){$\centerdot$}
\put(70,10){$\centerdot$}
\put(70,30){$\centerdot$}

\put(80,20){$\centerdot$}
\put(30,5){$\centerdot$}
\put(40,5){$\centerdot$}
\put(30,15){$\centerdot$}
\put(40,15){$\centerdot$}
\put(30,25){$\centerdot$}
\put(40,25){$\centerdot$}
\put(30,35){$\centerdot$}
\put(40,35){$\centerdot$}
\put(90,20){$\centerdot$}

\put(10,11){\vector(0,1){19}}
\put(10,30){\vector(4,-1){20}}
\put(21,11){\vector(2,3){9}}
\put(20,30){\vector(-1,0){9}}
\put(30,6){\vector(-2,1){9}}
\put(31,25){\vector(1,-1){9}}
\put(30,15){\vector(-2,-1){9}}
\put(31,35){\vector(0,-1){9}}

\put(40,5){\vector(-1,0){9}}
\put(41,26){\vector(2,1){9}}
\put(41,35){\vector(2,-1){9}}

\put(50,11){\vector(-2,1){9}}
\put(51,30){\vector(1,0){9}}

\put(61,11){\vector(0,1){19}}
\put(61,30){\vector(-3,-2){21}}

\put(81,20){\vector(1,0){9}}

\put(71,11){\vector(1,1){9}}
\put(71,30){\vector(1,-1){9}}

\put(35,20){\oval(60,40)}
\put(80,20){\oval(30,40)}

\end{picture} 
\caption{\small Arcs of $F$ are depicted, coming from the vertices of its two trees with roots at $y$ and $x$.}
\label{f}
\end{center}
\end{figure}
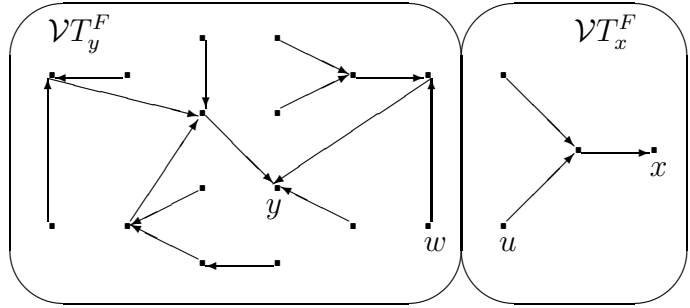

\subsection{Descendants among pseudo-descendants}

Our goal is to build minimal related forests using the \cite{V8} algorithm, which builds arbitrary pseudo-related forests of minimal weight. We need to make sure that the minimal forest $F\in\tilde{\cal F}^{k+1}$ has a descendant in the selected class of minimal pseudo-descendants ($\tilde{\cal F}^{k}\cap{\cal P}^F_y$). 

\begin{remark}
In fact, the indicated property is contained in the structure of proof the   Property 3 \cite[Theorem 2 (on related forests)]{V6}: taking two arbitrary forests $F\in\tilde{\cal F}^{k+1}$ and $G\in\tilde{\cal F}^{k}$ and, using only them, both the minimal descendant of the forest $F$ and the minimal ancestor of $G$ are constructed. We need not only to find a descendant in the selected subset of pseudo-descendants, but also to specify the exact transformations that are necessary to transform a pseudo-descendant into a descendant.   
\end{remark}

\begin{figure}[h]
\unitlength=1mm
\begin{center}
\begin{picture}(100,41)

\put (10,35){${\cal V}T^F_y$}
\put (80,35){${\cal V}T^G_x$}
\put (57,15){${\cal D}$}

\put (38,15){$y$}
\put (60,7){$w$}
\put (70,7){$u$}
\put (90,17){$x$}

\put(10,10){$\centerdot$}
\put(10,30){$\centerdot$}
\put(20,10){$\centerdot$}
\put(20,30){$\centerdot$}
\put(50,10){$\centerdot$}
\put(50,30){$\centerdot$}
\put(60,10){$\centerdot$}
\put(60,30){$\centerdot$}
\put(70,10){$\centerdot$}
\put(70,30){$\centerdot$}

\put(80,20){$\centerdot$}
\put(30,5){$\centerdot$}
\put(40,5){$\centerdot$}
\put(30,15){$\centerdot$}
\put(40,15){$\centerdot$}
\put(30,25){$\centerdot$}
\put(40,25){$\centerdot$}
\put(30,35){$\centerdot$}
\put(40,35){$\centerdot$}
\put(90,20){$\centerdot$}

\put(10,10){\vector(-1,0){10}}
\put(10,30){\vector(-1,0){10}}
\put(20,10){\vector(-1,0){9}}
\put(20,30){\vector(-1,0){9}}
\put(30,6){\vector(-2,1){9}}
\put(30,26){\vector(-2,1){9}}
\put(30,15){\vector(-2,-1){9}}
\put(30,35){\vector(-2,-1){9}}

\put(41,6){\vector(0,1){9}}
\put(41,26){\vector(2,1){9}}
\put(41,15){\vector(4,-1){19}}
\put(41,35){\vector(4,-1){19}}

\put(51,10){\vector(1,0){9}}
\put(51,30){\vector(1,-1){19}}

\put(61,10){\vector(1,0){9}}
\put(61,30){\vector(1,0){9}}

\put(81,20){\vector(1,0){9}}

\put(71,11){\vector(1,1){9}}
\put(71,30){\vector(1,-1){9}}

\put(35,20){\oval(60,40)}
\put(65,21){\oval(60,40)}
\put(50,10){\oval(30,18)}

\end{picture} 
\caption{\small Arcs of $G$ (which is a pseudo-descendant of $F$) coming from vertices of the set ${\cal V}T^F_y\cup {\cal V}T^G_x$ are depicted. ${\cal D}$ is the set of vertices of connected component of the induced subgraph $G|_{{\cal V}T^F_y}$, which includes the vertex $y$. Arcs of the forest $H$ are the arcs of $G$ coming from the vertices of the set ${\cal V}T^F_y$.}
\label{poda}
\end{center}
\end{figure}
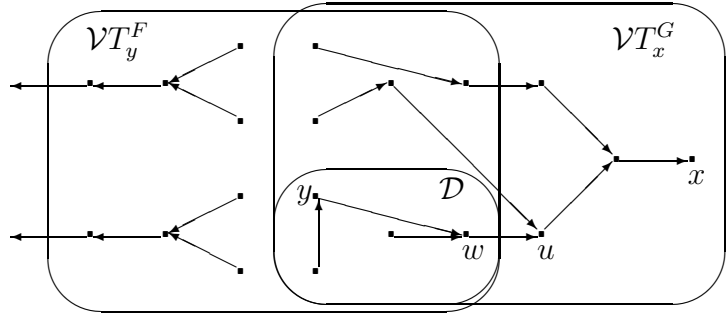

\begin{theorem}
Let $F\in\tilde{\cal F}^{k+1}$, $y\in{\cal K}_F$,  $G\in\tilde{\cal F}^{k}\cap {\cal P}^F_y$. 
Let also ${\cal D}$ be  
the set of vertices of that connected component of the induced subgraph $G|_{{\cal V}T^F_y}$ that contains vertex $y$. 
Then graph $F^G_{\uparrow{\cal D}}$ is a minimal descendant of $F$. 
\end{theorem}
\begin{proof}
Let us make sure that graphs $F$, $G$ and the set ${\cal D}$ satisfy conditions of Property 1.
Let the vertex $y$ in $G$ belong to a tree rooted at vertex $x$: $y\in{\cal V}T^G_x$. Graph $U=G|_{{\cal V}T^F_y}$ is a forest (in general, any subgraph of a forest is a forest), and its connected components are trees. Let $w$ be a root of that tree of $U$ to which vertex $y$ belongs:   $y\in{\cal D}={\cal V}T^U_w$.    
In $G$ itself, an arc comes from the vertex $w$ (in $G$, arcs come from all vertices of the set ${\cal V}T^F_y$) the entry of which does not belong to the set ${\cal V}T^F_y $ (in Fig. \ref{poda} this is arc $(w,u)$). This is the only arc coming in $G$ from the set ${\cal D}$. Thus, ${\cal N}^{out}_{\cal D}(G)\subset{\cal V}T^G_x\setminus{\cal V}T^F_y $.

Further, in $G$, since it is a pseudo-descendant of  $F$, no arcs come from the set $\overline{{\cal V}T^F_y}$ at all. Thus, in $G$ there are no arcs coming from $\overline{{\cal V}T^F_y}$ and entering ${\cal D}\subseteq {\cal V}T^F_y$. Since ${\cal D}$ is the set of vertices of the connected component of the induced subgraph $G|_{{\cal V}T^F_y}$, then in $G$ there are no arcs coming from ${\cal V} T^F_y \setminus {\cal D}$ and entering ${\cal D}$. Thus,
 ${\cal N}^{in}_{\cal D}(G)=\emptyset$. Considering also that 
 ${\cal D}\subset{\cal V}T^F_y\cap{\cal V}T^G_x$, we obtain from Property 1 that the graphs $F_{\uparrow\cal D}^G$ and $G_ {\uparrow\cal D}^F$ are forests.
 
 The set ${\cal D}$ itself contains exactly one root of $F$ (vertex $y$) and no roots of $G$, so $R=F^G_{\uparrow\cal D} \in { \cal F}^k$ and $G^F_{\uparrow\cal D}\in {\cal F}^{k+1}$. But then the conditions of Property 2 are also satisfied. Therefore, in particular,  $R\in\tilde{\cal F}^{k}$.
 
Let's check that $R$ is a descendant of $F$. Really, 
$R\in{\cal P}^F_y$, since forests $F$ and $R$ differ in no more than arcs coming from vertices of the set ${\cal V}T^F_y$ (more precisely, from the vertices of the set ${\cal D}\subseteq{\cal V}T^F_y$) and in ${\cal V}T^F_y$ there are no roots of  $R$. By construction in $R$, a single arc comes from the set ${\cal V}T^F_y$ (in the example in Fig. \ref{poda} and \ref{pods} this is the arc $(w,u)$). The outcome $w$ of this arc is the root of the tree $R|_{{\cal V}T^F_y}$.
\end{proof} 

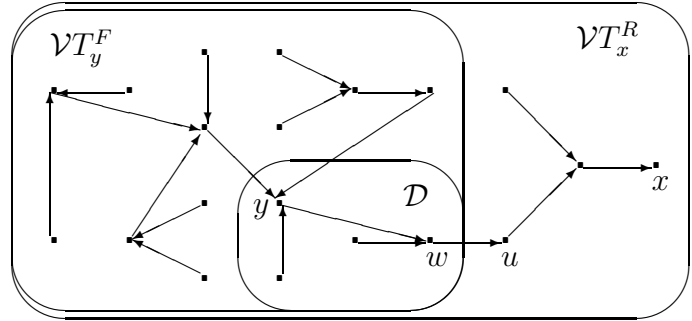
\begin{figure}[h]
\unitlength=1mm
\begin{center}
\begin{picture}(100,42)

\put (10,35){${\cal V}T^F_y$}
\put (80,36){${\cal V}T^R_x$}
\put (57,15){${\cal D}$}

\put (37,14){$y$}
\put (60,7){$w$}
\put (70,7){$u$}
\put (90,17){$x$}

\put(10,10){$\centerdot$}
\put(10,30){$\centerdot$}
\put(20,10){$\centerdot$}
\put(20,30){$\centerdot$}
\put(50,10){$\centerdot$}
\put(50,30){$\centerdot$}
\put(60,10){$\centerdot$}
\put(60,30){$\centerdot$}
\put(70,10){$\centerdot$}
\put(70,30){$\centerdot$}

\put(80,20){$\centerdot$}
\put(30,5){$\centerdot$}
\put(40,5){$\centerdot$}
\put(30,15){$\centerdot$}
\put(40,15){$\centerdot$}
\put(30,25){$\centerdot$}
\put(40,25){$\centerdot$}
\put(30,35){$\centerdot$}
\put(40,35){$\centerdot$}
\put(90,20){$\centerdot$}

\put(10,11){\vector(0,1){19}}
\put(10,30){\vector(4,-1){20}}
\put(21,11){\vector(2,3){9}}
\put(20,30){\vector(-1,0){9}}
\put(30,6){\vector(-2,1){9}}
\put(31,25){\vector(1,-1){9}}
\put(30,15){\vector(-2,-1){9}}
\put(31,35){\vector(0,-1){9}}

\put(41,6){\vector(0,1){9}}
\put(41,26){\vector(2,1){9}}
\put(41,15){\vector(4,-1){19}}
\put(41,35){\vector(2,-1){9}}

\put(51,10){\vector(1,0){9}}
\put(51,30){\vector(1,0){9}}

\put(61,10){\vector(1,0){9}}
\put(61,30){\vector(-3,-2){21}}

\put(81,20){\vector(1,0){9}}

\put(71,11){\vector(1,1){9}}
\put(71,30){\vector(1,-1){9}}

\put(35,21){\oval(60,40)}
\put(50,21){\oval(90,42)}
\put(50,11){\oval(30,20)}

\end{picture} 
\caption{\small Arcs of the forest $R=F^G_{\uparrow{\cal D}}$ (which is a descendant of $F$) coming from the vertices ${\cal V}T^F_y\cup {\cal V}T^G_x$ are depicted. Arcs coming from vertices of $\overline{\cal D}$ coincide for forests $R$ and $F$.}
\label{pods}
\end{center}
\end{figure}

\subsection{Place of the Chu-Liu-Edmonds algorithm}

The Chu-Liu-Edmonds algorithm \cite{china},\cite{Edmonds} (further CLE) is the only algorithm (up to different implementations \cite{Tar} and \cite{GGST}) for constructing a spanning directed tree of minimum weight with a given root. \footnote{Note that the CLE algorithm was originally formulated for an outgoing spanning tree (also called optimal branching), while our consideration concerns incoming forests and trees. To replace one type with another, it is enough simply transpose the weighted adjacency matrix.} The proposed algorithm (as well as the \cite{V8} algorithm) uses the CLE algorithm, and also, in particular, builds a spanning tree of minimum weight (regardless of the root). Of course, it cannot be considered a variant of the CLE algorithm: the latter is used as a component, and whatever implementation of it is taken will appear as a subroutine. 

All the above weight minima with an arbitrary selected set of vertices ${\cal S}$ are calculated using the CLE algorithm \cite{V8}. For our purposes, taking into account Property 3 (\ref{mlu}), only two of them are actually used: $\lambda_{\cal S}^\bullet$ and $\mu_{\cal S}^\circ$. 
And although calculating the values of $\lambda_{\cal S}^\bullet$ requires assigning each vertex of the set ${\cal S}$ sequentially to the root, the proposed algorithm is designed in such a way that direct application of the CLE algorithm can be avoided to determine such minima.
 The values $\lambda_{\cal S}^\bullet$ are determined from a variant of the formula (\ref{p}). But the values  $\mu_{\cal S}^\circ$ really have to be counted. 
 But their convenience lies in the fact that the problem of “blindness” of the CLE algorithm regarding the root (for it to work, the root must be assigned) does not exist for them --- it is known what should be the root.
Indeed \cite{V8} to determine the minima of $\mu^\circ_{\cal S}$, vertices from $\overline{\cal S}$ merge into one vertex $*$.
Special graph $W$ is constructed on the set of vertices ${\cal S}\cup\{ *\}$. The set of its arcs consists of all $(q,p)\in{\cal A}V$ for $q,p\in{\cal S}$, with the same weights $w_{qp}=v_{qp}$ , as well as from arcs $(q,*)$ with weights 

\begin{equation}
w_{q*}=\min_{p\in\overline{\cal S}}v_{qp}, \ q\in{\cal S} \ . 
\label{ficsd}
\end{equation}
Next, using the CLE algorithm, a minimum spanning tree of the graph $W$ is constructed with the root at vertex $*$. According to it, the corresponding forest from 
${\cal F}^\circ_{\cal S}$ is restored which gives a minimum in (\ref{cf}).

\section{Bypassing the stage of pseudo-descendants to descendants}

The justification for constructing minimal pseudo-descendants \cite[Theorem 2]{V8} let us formulate as the following  
\begin{property}
{\it Let $F\in \tilde{\cal F}^{k+1}$, $y\in{\cal K}_F$ and $G\in{\cal P}^F_y$. Then, in order for $G\in \tilde{\cal F}^k$, it is necessary and sufficient that 
\begin{equation}
\Upsilon^G_{{\cal V}T^F_y}=\mu_{{\cal V}T^F_y}^\circ
\label{ul}
\end{equation}
 and    
\begin{equation}
\mu_{{\cal V}T^F_y}^\circ- 
\mu^\bullet_{{\cal V}T^F_y} = \min_{q\in{\cal K}_F} \left( \mu_{{\cal V}T^F_q}^\circ- 
\mu^\bullet_{{\cal V}T^F_q}  \right)  \ . 
\label{potom}
\end{equation}}
\end{property}

\begin{remark}
Due to the equality of forest-like (type $\mu_{\cal S}^{\bullet}$) and tree-like (type $\lambda_{\cal S}^{\bullet}$) minima at $ {\cal S}$ coinciding with the set of vertices of any tree of any minimal forest, it becomes arbitrary which type of minimum to indicate in the formulas. It seems natural to indicate the one that was actually considered.
Namely, if the spanning forests ${\cal F}_{\cal S}^{\bullet}$ were used in the calculations, then specify $\mu_{\cal S}^{\bullet}$. If trees from the sets ${\cal T}_{\cal S}^{\bullet}$ were used, then specify $\lambda_{\cal S}^{\bullet}$. Nevertheless, (\ref{potom}) contains the values of $\mu^\bullet_{\cal S}$ for the convenience of writing the formula itself. In practice, both in meaning and in calculations, it is precisely the quantities $\lambda^\bullet_{\cal S}$ that are found. 
\end{remark}

\subsection{Constructing descendants bypassing the stage of pseudo-descendants}

Combining Theorem 1 and Property 5, we obtain a method for constructing a minimal descendant bypassing the pseudo-descendant stage.

\begin{theorem}
Let $F\in\tilde{\cal F}^{k+1}$, $y\in{\cal K}_F$ and satisfied (\ref{potom}). Let also $H\in{\cal F}^\circ_{\cal Y}$, where ${\cal Y}={\cal V}T^F_y$, and
\begin{equation}
\mu_{\cal Y}^\circ=\Upsilon^H_{\cal Y} \ , 
\label{muh}
\end{equation}
and ${\cal D}$ is the set of vertices of a connected component of the graph $H|_{\cal Y}$ containing vertex $y$.  Then the graph $R=F^H_{\uparrow{\cal D}}$ is a minimal descendant of the forest $F$  
and $\Upsilon^R_{\cal Y}=\mu_{\cal Y}^\circ$. 
\end{theorem} 

\begin{proof}
By construction, the graph $G=F^H_{\uparrow{\cal Y}}$ is a pseudo-descendant of $F$ and $\Upsilon^G_{\cal Y}=\mu^\circ_{\cal Y}$. In this case, (\ref{potom}) continues to be executed, since it concerns only the graph $F$. Then $G$ satisfies the conditions of Property 5, so it is minimal. Since $G|_{\cal Y}=H|_{\cal Y}$, then ${\cal D}$ is the set of vertices of the connected component of the graph $G|_{\cal D}$ containing the vertex $y$. Then, by Theorem 1, the graph $R=F^G_{\uparrow{\cal D}}$ is a minimal descendant of the forest $F$. 
And since the same arcs come from the vertices of the set ${\cal D}$ in the forests $G$ and $H$, then $R=F^H_{\uparrow{\cal D}}$.

To conclude the proof, we note that in graphs $R$ and $G$ the arcs coming from the vertices of the set ${\cal Y}\setminus {\cal D}$ are, generally speaking, different, nevertheless
\begin{equation*}
\Upsilon^R_{\cal Y}=\Upsilon^G_{\cal Y} \ . 
\end{equation*} 
Indeed, both forests $R$ and $G$ are minimal. The graph $R$, being a descendant, in particular, is also a pseudo-descendant of $F$, as is the graph $G$ (both of them belong to the set ${\cal P}^F_y$). Thus, Property 5 holds for them (\ref{ul}). 
\end{proof}

If $F$ is some minimal forest, then for each of its trees $T^F_q$, $q\in{\cal K}_F$, the weights  $\lambda_{{\cal V}T^F_q}^\bullet =\mu^\bullet_{{\cal V}T^F_q}= \Upsilon^{T^F_q}$ are automatically known. Let the quantities $\mu_{{\cal V}T^F_q}^\circ$ be known too, and let $y$ be the root for which (\ref{potom}) holds. 
 In the \cite{V8} algorithm, forest $H\in{\cal F}^\circ_{\cal Y}$, ${\cal Y}={\cal V}T^F_y$, is found for which (\ref{muh}) holds, and a pseudo-descendant of the forest $F$ is constructed from it --- forest $G=F^H_{\uparrow{\cal Y}}$. Theorem 2 allows for given $H$, to immediately construct a minimal descendant without creating a pseudo-descendant. But although the need to construct a minimal pseudo-descendant disappears completely, the same preliminary procedure is necessary. Namely: under given $F$, for each root $q\in{\cal K}_F$, construct the graph $H^q\in{\cal F}^\circ_{\cal Q}$, ${\cal Q} ={\cal V}T^F_q$, for which $\Upsilon_{\cal Q}^{H^q}=\mu_{\cal Q}^\circ$, and determination of the  root $y$ satisfying (\ref{potom}).

\subsection{Finding a connected component and an absorbing tree}

To determine the set of vertices ${\cal D}$ of a connected component $H|_{\cal Y}$ containing a given vertex $y$, you can use both breadth-first search and depth-first search \cite{CLRS} in undirected graph.
If graph has $N$ vertices and $M$ edges, then the complexity of determining all connected components by depth-first search is $O(N+M)$.
If graph is an undirected forest (acyclic graph) with $k$ components, then the number of edges in it is $(N-k)$ and the total complexity is estimated to be $O(N)$.

We assume that in $H$ arcs originate only from the vertices of the set ${\cal Y}={\cal V}T^F_y$. At the same time, by definition, arcs outcome from all the vertices of the set ${\cal Y}$. Therefore, from the set of vertices of any connected component of  $H|_{\cal Y}$ in the graph $H$ itself, there is a single arc whose entry does not belong to the set ${\cal Y}$. This property allows the search itself to be performed not in the induced forest $H|_{\cal Y}$, but in the forest $H$ itself. In this case, the graph $H$  may even be connected, that is, it may be a tree with a root at some vertex $u\notin{\cal Y}$, and at the same time the graph $H|_{\cal Y}$ may have several components.

In our case, the forest $H$ is transformed into an acyclic graph by simply ignoring the orientation (the arcs are treated as edges). The weights of the arcs (edges) do not matter.  It is required to construct only one component of the induced forest $H|_{\cal Y}$ containing vertex $y$. To do this, depth-first (or breadth-first) search is started from vertex $y$ in the forest $H$ itself. 
Moreover, as soon as an arc $(w,u)$ appears in the search, the entry of which (vertex $u$) does not belong to the set ${\cal Y}$ (and this arc is unique), then this arc and vertex $u $ are removed from the constructed fragment and the search continues.  After completing the procedure, the set ${\cal D}$ is constructed. This set is the set of vertices of the required component of the induced forest $H|_{\cal Y}$.
The vertex $u$, which does not belong to it, is also highlighted. 
The time to construct the required component does not exceed $O(|{\cal Y}|)$.

If each vertex is provided with an index pointing to the root of the tree in the forest $F$ under study to which it belongs, then the appearance of the vertex $u\notin{\cal Y}$ will be detected immediately, and the “absorbing” tree will also be immediately detected.  
This is some tree $T^F_x$: $u\in{\cal V}T^F_x$.
It is this tree that absorbs the modified tree $T^F_y$, forming the tree $T^R_x$ of the minimal descendant $R=F^H_{\uparrow{\cal D}}$ of the forest $F$.
Since $R$ is minimal, then by Property 3 $\mu^\bullet_{\cal Z}=\lambda^\bullet_{\cal Z}= \Upsilon^R_{\cal Z}$ holds, where $ {\cal Z}={\cal V}T^R_x$. Therefore the expression (\ref{p}) takes the form

\begin{equation}
\mu^\bullet_{\cal Z}=\mu^\bullet_{\cal X}+\mu^\circ_{\cal Y} \ , \ \ {\cal X}={\cal V}T^F_x, \  {\cal Y}={\cal V}T^F_y \ , \ \ {\cal Z}={\cal X}\cup{\cal Y} \ .
\label{Z}
\end{equation}

In fact, when constructing a descendant, only arcs coming from the vertices of the set ${\cal D}$ are replaced. That's why

\begin{equation}
\mu_{\cal Y}^\circ=\Upsilon^F_{{\cal Y}\setminus{\cal D}}+\Upsilon^H_{\cal D}=\Upsilon^F_{\cal Y} -\Upsilon^F_{\cal D}+\Upsilon^H_{\cal D}=\mu^\bullet_{\cal Y} -\Upsilon^F_{\cal D}+\Upsilon^H_{\cal D} . 
\label{Y}
\end{equation}

In example in Fig. \ref{poda} the forest $H$ consists of arcs coming in $G$ from the vertices of the set ${\cal Y}$. 
The arc $(w,u)$ is the only arc that in $H$ has an outcome in ${\cal D}$ and an end in $\overline{\cal D}$. In the example, $u\in{\cal V}T^F_x$ and it is this tree that absorbs the modified tree $T_y^F$ (see also Fig. \ref{f} and Fig. \ref{pods}). Moreover, the set ${\cal D}$ is the set of vertices of the required connected component of the induced forest $H|_{\cal Y}$.

\section{Algorithm for constructing minimal descendants}
\subsection{Using simplified notation}
 
In what follows, only the weights of sets of tree vertices of the minimal forest $F$ under study are used.
Therefore, we will adopt the following convention, which shortens the designation of the weights of sets. Instead of the entire set of vertices of the tree $T^F_q$, we will specify only its root as the subscript. That is, taking into account (\ref{mlu}):  

\begin{equation}
	\lambda^\bullet_q=\mu^\bullet_q\equiv\mu^\bullet_{{\cal V}T^F_q}=\lambda^\bullet_{{\cal V}T^F_q}, \  \  \mu_q^\circ\equiv\mu^\circ_{{\cal V}T^F_q},  \ \ q\in{\cal K}_F \ .
	\label{lm}
\end{equation}
\subsection{Description of the algorithm}

Since we are considering minimization problems, we accept the following convention. If there is no arc $(q,p)$ in the graph, then the corresponding weight is set equal to infinity: $v_{qp}=\infty$.

Let us describe an algorithm in which, in accordance with Property 5 (which is the rationale for the algorithm for constructing minimal pseudo-descendants in \cite{V8}) and Theorem 2, related forests are constructed.

Since it is necessary to monitor the fulfillment of conditions (\ref{potom}) and (\ref{muh}), we select two lists separately: list $\bold\Upsilon$ of non-infinity forest-like minima $\mu_{q}^\circ$ , $q\in{\cal K}_F$, and an ascending list of $\bold\Delta$ increments $\mu_{q}^\circ- \mu^\bullet_{q}$. We also introduce labels: $x$ --- the name of the root under study, $y$ --- the name of the first vertex of the ordered list $\bold\Delta$. 

The algorithm starts with an empty spanning forest $F\in{\cal F}^N=\tilde{\cal F}^N$ (this forest is the only element of the set ${\cal F}^N$). For it $\mu_{q}^\circ=\mu^\bullet_{q}=0$,  $q\in{\cal K}_F={\cal N}$, ${\bold\Upsilon}={\bold\Delta }=\emptyset$.

1st step.

$a)$ For each vertex $q$ there is an arc $(q,p_q)$ of minimum weight coming from it. We put $\mu_{q}^\circ=v_{qp_q}=\underset{p\neq q}{\min} v_{qp}$. 
Spanning forests $H^q$ are created, consisting of one arc: ${\cal A}H^q=\{ (q,p_q) \}$. 

$b)$ The values of $\mu_q^\circ$ other than $\infty$ are entered in the list of $\bold\Upsilon$. If this list is empty (which corresponds to an empty source graph), the algorithm stops working. Otherwise, the increments $(\mu_q^\circ-\mu^\bullet_q)=\mu_q^\circ$ are entered in the list $\bold\Delta$.

$c)$ The list $\bold\Delta$ is sorted in ascending order. 

$d)$ The vertex corresponding to the first element in the list $\bold\Delta$ receives the label $y$. The entry of an arc coming from $y$ at forest $H^y$ receives the label $x$ (the root under study is assigned): $\mu_y^\circ=v_{yx}$.

$e)$  The arc $v_{yx}$ is added to the empty forest. The number of trees decreased by one. The resulting graph becomes the new forest $F$ under study: formally we set $F:=F^{H^y}_{\uparrow\{ y\}}$ (At the first step, this forest simply coincides with $H^y$). 

The weight $\mu^\bullet_x:=\mu^\bullet_x+\mu_{y}^\circ=0+v_{yx}=v_{yx}$ is updated. The set of vertices of the empty tree consisting of vertex $x$ has increased --- a new vertex has been added --- vertex $y$.

$f)$ The values with the label $y$ (the vertex of $y$ has ceased to be the root) and with the label $x$ (they need to be adjusted) are removed from the lists $\bold\Upsilon$ and $\bold\Delta$. 

Further steps $n=2,3, \ldots , N-1$ occur in a recurrent manner.

$n$-th step. 

The forest under study $F$ consists of $(k+1)$ trees ($k=N-n$). The weights $\mu^\bullet_{p}$ are known for all its trees (indexed by the roots).  For all roots except the root under study $x$, the values $\mu_{p}^\circ$ that make up the list  $\bold\Upsilon$ are known. The ordered list of $\bold\Delta$ increments is also incomplete --- there is no information about the increment associated with the root $x$ under study.

$a)$ For a tree with the root under study (label $x$), the value $\mu_x^\circ$ is determined using the CLE algorithm (see paragraph 3.3 of this article) and the spanning forest $H^x\in{\cal F}^\circ_{\cal X}$, ${\cal X}={\cal V}T^F_x$ is found, for which $\Upsilon^{H^x}=\mu_{x}^\circ$. 

$b)$ If there is no such forest ($\mu_x^\circ=\infty$), the list $\bold\Upsilon$ is checked. If $\bold\Upsilon=\emptyset$ the algorithm is interrupted (the constructed forest $F$ has the minimum possible number of trees). Otherwise, go to point $d)$. 

If $\mu_x^\circ<\infty$, go to point $c)$.  

$c)$ The value $\mu_x^\circ$ is entered into the list $\bold\Upsilon$.  Increment $(\mu_{x}^\circ- 
\mu^\bullet_{x})$ is entered into the ordered list $\bold\Delta$ and its place in this list is determined.

$d)$ The root corresponding to the first element of the list $\bold\Delta$ receives the label $y$. The value $\mu_{y}^\circ$ corresponds to the spanning forest $H=H^y\in{\cal F}^\circ_{\cal Y}$, ${\cal Y}={\cal V}T ^F_y$, defined in point $a)$ at some previous step or at the 1st step: $\Upsilon^{H}=\mu_{y}^\circ$. 

$e)$ 
In the forest $H$, a depth-first (or breadth-first) search is launched starting at vertex $y$. As soon as the search encounters an arc $(w,u)$ whose entry (vertex $u$) does not belong to the set ${\cal Y}$, this arc and vertex $u$ are removed from the search and the search continues until the set of vertices ${\cal D}$ is formed of the desired connected component (see Section 4.3 of this article).
 The vertex $u$ belongs to some tree of the forest $F$. The root of this tree receives the label $x$ --- a  new root under investigation is assigned. 

$f)$ By Theorem 2, the graph $F^{H}_{\uparrow{\cal D}}$ is the minimal descendant of the forest $F$, consisting of $k$ trees. We set $F:=F^{H}_{\uparrow{\cal D}}$.  

In accordance with (\ref{Z}), we set $\mu^\bullet_{x}:= \mu^\bullet_{x}+\mu_{y}^\circ$.  Now this weight corresponds to the set of vertices of the enlarged tree (the tree rooted at $x$ has absorbed the modified tree, whose root was vertex $y$ before the reorientation). 
   
 From the lists  $\bold\Upsilon$  and $\bold\Delta$ values are deleted with the index $y$ (the vertex of $y$ has ceased to be the root). Values with the label $x$ are also removed from these lists (they need to be recalculated, since now $x$ is the root of the enlarged tree). \footnote{Note that if ${\bold\Upsilon}=\emptyset$ is at this point in the algorithm, the algorithm's operation is not interrupted, since the value of $\mu_x^\circ$ has not yet been determined and added to the list.}  
 
$g)$ $n:=n+1$ (step number increases by one). 
If $n=N-1$ the algorithm terminates (the last forest constructed is the minimum weight spanning tree. Its root is the vertex labeled $x$). 
Otherwise, go to the $n$-th step.

The algorithm stops working (see point $b)$ of the recurrent procedure) when a forest consisting of the minimum possible number of trees has been built.  If the source graph has at least one spanning tree, then the output of the algorithm will be a set of $N$ minimal spanning related forests --- one representative each, consisting of $k=1,2,\ldots, N$ trees.

To remember spanning forests (this point is not described in the algorithm for brevity), it is enough to enter into the corresponding list only the arcs coming from the vertices of ${\cal D}$ (point $e$). Arcs coming from other vertices of the minimal forest under construction are the same as in the previous one.

\subsection{Algorithm complexity}

A component of the proposed algorithm is the CLE algorithm for finding a minimum spanning tree with a given root. In the case of dense graphs ($M\sim N^2$, where $N$ is the number of vertices and $M$ is the number of arcs), its most efficient implementations \cite{Tar} and \cite{GGST} have complexity $O(N^2)$.

When estimating complexity, we assume that the original graph $V$ is dense.

$1$th step. 

$a)$ 
The total complexity of determining the minimum element for the adjacency list of each vertex is $O(N^2)$. 

$c)$  It takes $O(N\log N)$ operations to arrange a list $\bold\Delta$ consisting of no more than $N$ values.

$n$-th step.

$a)$  At the $n$th step, the forest $F$ contains a $k+1$ trees ($k=N-n$). Then any tree contains at most $n$ vertices. Only one weight is calculated, the weight of $\mu_{x}^\circ$. The CLE algorithm is applied with the introduction of a generalized vertex $*$ combining the vertices of $\overline{\cal X}$ into a single vertex (see point 3.3). The calculation takes $O(|{\cal X}|^2+1)^2$ operations. Since $|{\cal X}|\leq N-k=n$, the complexity of the most costly situation (when a single tree grows all the time) is not higher than $O(N-k+1)^2=O(n+1)^2$. 

$c)$ There are no more than $k$ values in the ordered list of $\bold\Delta$. Adding the increment $\mu_{x}^\circ -\mu^\bullet_{x}$ to this list and finding its place in it takes $O(\log k)$ comparison actions. In terms of the step number, this is $O(\log(N-n))$.

$e)$ A depth-first search to determine the connected component of the forest $H|_{\cal Y}$, which contains
the vertex $y$, takes $O(|{\cal Y}|)$ operations.  Since $|{\cal Y}|\leq n$, the complexity does not exceed $O(n)$. 

The main computational costs are associated with step $a)$ of the recurrent procedure, where the CLE algorithm is applied. The complexity of the $n$-th step (except for the starting one) does not exceed estimates of $O(n^2)$.  Summing over all steps, we find that the total complexity does not exceed $O(N^3)$.

\subsection{Final notes on the algorithm}

The proposed algorithm has the same complexity as the \cite{V8} algorithm, but it has several advantages. Firstly, the forests under construction are related, that is, they differ from each other in the least possible way. Secondly, the CLE algorithm, which is part of both algorithms, is applied in the new version only once at each step (since at each step the disappearing tree is absorbed by exactly one other tree, it is applied only to this last one). Note also that when constructing a minimum spanning tree, the complexity of all three algorithms is the same. The difference is that the CLE algorithm with a dedicated root has to be applied $N$ times, changing the root each time. That is, it actually produces $N$ minimal spanning trees with a root at each vertex, from which the minimum one is then selected.  The proposed algorithm produces a sequence of minimal related forests with a different number of trees (up to a spanning tree), regardless of which vertices turn out to be the roots. The roots are determined by the algorithm itself.

\section{Analogy of the structure of algorithms with stochastic processes}

The CLE algorithm (as applied to entering  forests) starts with the fact that an arc of minimum weight is emitted from each vertex except the root. The vertices that form one cycle or another condense into a single vertex. For these enlarged vertices, the weights of the new transition arcs are determined, and for the new graph, a minimum weight arc is emitted from each vertex except the root. Each of the cycles that occurs in this case is combined back into one vertex, and so on until the sequence of enlargements leads to a tree. Then, the system of nested loops unfolds in the opposite direction: an arc is removed from the loop, the outcome of which is the outcome of the arc of the intermediate tree.  Thus, the CLE algorithm, despite the fact that it builds a spanning tree with a dedicated root, indicates more about the cyclic structure of the graph than about its tree structure. And if you are not interested in the spanning tree itself, but in how it is obtained and from which substructures, then the system of nested loops is an indicator.

If we consider a stochastic differential equation
\begin{equation}
\dot x_t^\varepsilon=b(x^\varepsilon_t)+ \varepsilon\dot w_t \ , 
\label{win}
\end{equation}
where $w_t$ is a Wiener process, $\varepsilon$ is a small parameter, then at exponentially large times consistent with the small parameter ($t(\varepsilon)=\tau \exp(K/\varepsilon)$) the process is significantly turns into a Markov chain with a finite number of states (their number is equal to the number of areas of attraction of the dynamical system $\dot x_t=b(x)$). The individual trajectories of this process mostly wander with a dedicated circumvention direction through some cycles consisting of some areas of attraction. Belonging to a particular cycle depends on the starting point. As the time scale increases (increasing $K$), a set of cycles of the first rank passes into cycles of the next rank, in which cycles of the previous level are perceived as a single area. Etc. A hierarchy of nested loops \cite{VF} is created. In this sense, if we look at the CLE algorithm itself as a process, it "simulates" \ the preferential wandering of individual process trajectories (\ref{win}) and indicates which cycles these trajectories will follow depending on the time scale and starting point.

The infinitesimal operator of the semi-group generated by the transition function of the Markov process (\ref{win}) is formally conjugate to the diffusion operator and appears on the right side of the Fokker-Planck equation, which describes the evolution of one-dimensional distributions of this process. 
At large times, diffusion is well approximated by a Markov chain \cite{V1}-\cite{V2}.   
The row vector of the distribution density $\vec{\bf x}^\varepsilon$ satisfies the system
\begin{equation}
\frac{d \vec {\bf x}^\varepsilon}{d t}=-\vec{\bf x}_0^\varepsilon L^\varepsilon \ ,
\label{ode}
\end{equation}    
where $L^\varepsilon$ is a matrix with off-diagonal elements of the form: $L^\varepsilon_{ij}=l_{ij}\exp(-v_{ij}/ 
\varepsilon^2)$. 
The values of $v_{ij}$ have the meaning of working against forces of the drift field $b(x)$ during the transition from one region of attraction to the neighboring one. 
Eigenvalues of  matrix $L^\varepsilon$ are exponentially small and have the order $\exp((\phi^{k+1}-\phi^{k})/\varepsilon^2)$, $k=1,2,\ldots ,N$. When exponential order of time $K$ passes through the corresponding increments $(\phi^{k}-\phi^{k+1})$, nature of the behavior (\ref{ode}) changes and sub-limit distributions are formed --- distributions that are invariant on a given time scale. These distributions are concentrated in the roots of forests from $\tilde{\cal F}^k$, and trees with these roots represent areas of attraction of the corresponding roots. Thus, the steps of algorithm for constructing minimal forests “simulate” the evolution of one-dimensional distributions and the formation of invariant measures in sub-processes corresponding to different time scales $t(\varepsilon)$.  Moreover, the algorithm calculates spectral characteristics of the dynamics (\ref{ode}). But its scope is much wider. From Big-Data structuring to the general problem of clustering --- on what principle are clusters formed, what goes into a cluster and what are the connections between clusters. 

\centerline{Abstract}

\begin{center}{Algorithm for Constructing Related Spanning Directed Forests of Minimum Weight}
\end{center}

\centerline{Buslov V.A.}

\parbox[t]{12cm}
{\small 
An algorithm is proposed for constructing directed spanning forests of the minimum weight, in which the maximum possible degree of affinity between the minimum forests is preserved when the number of trees changes. The correctness of the algorithm is checked and its complexity is determined, which does not exceed $ O (N ^ 3) $ for dense graphs. The result of the algorithm is a set of related spanning minimal forests consisting of $ k $ trees for all admissible $ k $.}

\vspace{0.5cm}

St. Petersburg State University, Faculty of Physics, Department of Computational Physics.

198504 St. Petersburg, Old Peterhof, st. Ulyanovskaya, 3.

Email: abvabv@bk.ru, v.buslov@spbu.ru

\end{document}